\documentclass[12pt]{amsart}
\usepackage{amsmath}
\usepackage{amssymb}
\usepackage{epsfig}
\usepackage{graphicx}
\usepackage{xcolor}
\numberwithin{equation}{section}

\usepackage[all,2cell]{xy}  

\calclayout
\allowdisplaybreaks[3]

\DeclareFontFamily{U}{mathb}{\hyphenchar\font45}
\DeclareFontShape{U}{mathb}{m}{n}{
      <5> <6> <7> <8> <9> <10> gen * mathb
      <10.95> mathb10 <12> <14.4> <17.28> <20.74> <24.88> mathb12
      }{}
\DeclareSymbolFont{mathb}{U}{mathb}{m}{n}
\DeclareMathSymbol{\righttoleftarrow}{3}{mathb}{"FD}

\newcommand{\eqto}{\stackrel{\lower1.5pt\hbox{$\scriptstyle\sim\,$}}\to}
\newcommand{\eqdashto}{\stackrel{\lower1.5pt\hbox{$\scriptstyle\sim\,$}}\dashrightarrow}

\theoremstyle{plain}
\newtheorem{prop}{Proposition}

\newtheorem{theo}[prop]{Theorem}
\newtheorem{coro}[prop]{Corollary}

\newtheorem{lemm}[prop]{Lemma}

\theoremstyle{definition}
\newtheorem{defi}[prop]{Definition}

\newtheorem{rema}[prop]{Remark}

\newtheorem{exam}[prop]{Example}

\makeatother
\makeatletter

\newcommand{\bA}{\mathbb A}
\newcommand{\bC}{\mathbb C}
\newcommand{\bF}{\mathbb F}

\newcommand{\bP}{\mathbb P}

\newcommand{\bR}{\mathbb R}
\newcommand{\bZ}{\mathbb Z}

\newcommand{\cA}{\mathcal A}

\newcommand{\cE}{\mathcal E}

\newcommand{\cJ}{\mathcal J}
\newcommand{\cK}{\mathcal K}

\newcommand{\cO}{\mathcal O}

\newcommand{\cR}{\mathcal R}
\newcommand{\cU}{\mathcal U}

\newcommand{\cX}{\mathcal X}
\newcommand{\cY}{\mathcal Y}

\newcommand{\rD}{\mathrm D}

\newcommand{\rH}{\mathrm H}
\newcommand{\rK}{\mathrm K}
\newcommand{\rL}{\mathrm L}

\newcommand{\fS}{\mathfrak S}

\newcommand{\fp}{\mathfrak p}
\newcommand{\fq}{\mathfrak q}

\newcommand{\Bl}{\operatorname{Bl}}

\newcommand{\Def}{\operatorname{Def
}}

\newcommand{\Gr}{\operatorname{Gr}}
\newcommand{\Pic}{\operatorname{Pic}}

\newcommand{\Aut}{\operatorname{Aut}}

\newcommand{\ra}{\rightarrow}

\author[B\"ohning]{Christian B\"ohning}
\address{Mathematics Institute\\
Zeeman Building\\
University of Warwick\\
Coventry CV4 7AL\\
UK}
\email{C.Boehning@warwick.ac.uk}

\author[von Bothmer]{Hans-Christian Graf von Bothmer}
\address{Fachbereich Mathematik\\
Bundesstraße 55\\
20146 Hamburg\\
Germany}
\email{hans.christian.v.bothmer@math.uni-hamburg.de}

\author[Tschinkel]{Yuri Tschinkel, with an appendix by Brendan Hassett}
\address{Courant Institute\\
                New York University \\
                New York, NY 10012 \\
                USA }
\email{tschinkel@cims.nyu.edu}

\address{Simons Foundation\\
160 Fifth Avenue\\
New York, NY 10010\\
USA}

\title[Equivariant birational geometry and derived categories]{Equivariant birational geometry of cubic fourfolds and derived categories}

\begin{document}
\date{\today}

\begin{abstract}
We study equivariant birationality from the perspective of derived categories. We produce examples of nonlinearizable but stably linearizable actions of finite groups on smooth cubic fourfolds. 
\end{abstract}

\maketitle

\section{Introduction}
\label{sect:intro}

Let $X$ be a smooth projective algebraic variety over a field $k$ and $\rD^b(X)$ its bounded derived category of coherent sheaves. It is a rich algebraic object:   
a celebrated theorem of Bondal and Orlov \cite{BO} states that $\rD^b(X)$ determines $X$ uniquely, if its canonical or anticanonical class is ample. 
This uniqueness can fail, there exist nonisomorphic but 
{\em derived equivalent} varieties, 
e.g., K3 surfaces or abelian varieties. 
These results and constructions inspired active investigations of derived categories, and derived equivalences in various contexts. 

In some sense,  $\rD^b(X)$ contains {\em too much} information, or rather, the data that are relevant in concrete geometric applications are hard to visualize. The overarching goal is to extract computable, more compact, invariants of derived categories that would allow to answer basic questions about geometry, such as
\begin{itemize}
\item existence of $k$-rational points, or
\item $k$-rationality. 
\end{itemize}
This has been pursued in, e.g., \cite{HT-derived}, 
\cite{auel-bernardara}, \cite{AB-surv}, \cite{AAFH}.


One natural candidate for an invariant of $\rD^b(X)$ 
is the {\em Kuznetsov component} $\cA_X$, an admissible 
subcategory of $\rD^b(X)$,
which however depends on the choice of a
maximal semiorthogonal decomposition (see Section~\ref{sect:derived} for definitions). 
The expectation is that this component captures, in particular, rationality properties of $X$; this idea has been tremendously influential. It has been tested in many situations, e.g., Fano threefolds, or special cubic fourfolds. In these cases, the Kuznetsov component is identified as the orthogonal to a naturally defined {\em exceptional sequence} of objects in $\rD^b(X)$. 

Due to its universality, one might expect that this approach is valid over nonclosed fields,  as well as in presence of group actions. 
Here, we explore this in detail in the equivariant context, for smooth cubic fourfolds, equipped with a regular, generically free action of a finite group $G$.  
Our main result is:

\begin{theo}
\label{thm:main}
There exist smooth Pfaffian cubic fourfolds $X$ 
with a regular generically free action of a finite group $G$
such that 
\begin{itemize}
\item the $G$-action is not linearizable, i.e., not equivariantly birational to a (projective) linear $G$-action on $\bP^4$, 
\item the $G$-action on $X\times \bP^1$, with trivial action on the second factor, is linearizable,
\item 
the standard Kuznetsov component $\cA_X$ is $G$-equivalent to $\rD^b(S)$, 
with the $G$-action induced by an embedding of $G$ into
the automorphisms of a K3 surface $S$,
\item the variety of lines $F_1(X)$ is $G$-birational to $S^{[2]}$, the  Hilbert scheme of two points on $S$. 
\end{itemize}
\end{theo}

A more precise version is given in Theorem~\ref{thm:GK3NotLin}.
Our main theorem contradicts natural equivariant analogs of existing rationality conjectures, as explained in Section~\ref{sect:geom}.

The stable linearizability proof is based on an adaptation to the equivariant context of the classical Pfaffian construction. This allows us to establish new stable linearizability results for, e.g., quadric surfaces, see Section~\ref{sect:examples}.

\

\

\noindent
{\bf Acknowledgments:} 
The first author was supported by the EPSRC New Horizons Grant EP/V047299/1.
The third author was partially supported by NSF grant 
2000099. 

We are grateful to Chunyi Li for helpful discussions, to Brendan Hassett for proving a result on $G$-Hodge structures, in the Appendix, and to Alexander Kuznetsov for several comments.

\section{Derived categories}
\label{sect:derived}

We recall basic notions concerning derived categories that are used in applications to birational geometry, see, e.g., \cite{Orlov-survey}. 

\subsection*{Notation}
Let $G$ be a finite group. 
A $G$-variety over $k$ is an algebraic variety with a regular action of $G$.  
From now on, by a {\em category} we mean a $k$-linear triangulated category. A $G$-category is a category $\cA$ together with a homomorphism 
$$
G\to \mathrm{Auteq}(\cA),
$$
the group of autoequivalences of $\cA$. A strictly full $k$-linear triangulated subcategory $\mathcal B\subset \cA$ of a $G$-category $\cA$ is $G$-stable if for every object $E\in \mathcal B$ and every $g\in G$ we have $g_*E \in \mathcal B$.  

\subsection*{Semiorthogonal decompositions}
Let $X$ be a smooth projective variety over a field $k$ and $\rD^b (X)$ its derived category of coherent sheaves. 
An object $E\in\rD^b (X)$ is called {\em exceptional} if 
\[
\mathrm{Hom}(E, E) \simeq k, \quad \mathrm{Ext}^r (E, E) = 0, \; r\neq 0.
\]
An exceptional sequence is an ordered tuple of exceptional objects 
\[
(E_1, \dots , E_n)
\]
such that 
\[
\mathrm{Ext}^l (E_r, E_s)=0, \quad \forall\, r>s, l .
\]
An exceptional sequence is called {\em full} if  the smallest full triangulated subcategory containing the $E_r$ is equivalent to $\rD^b (X)$, i.e., if the sequence generates $\rD^b (X)$. 

The notion of {\em semiorthogonal decomposition} generalizes the preceding concepts. 
A full subcategory $\cA$ of $\rD^b (X)$ is called {\em admissible} if the inclusion functor has a left and right adjoint. 
A sequence 
$$
(\cA_1, \dots , \cA_n)
$$
of admissible subcategories of $\rD^b(X)$
is called a {\em semiorthogonal decomposition} of $\rD^b (X)$ if
\begin{itemize}
\item the $\cA_1,\ldots, \cA_n$ generate $\rD^b (X)$ and 
\item there are no derived Hom's from any object in $\cA_r$ to an object in $\cA_s$, for $r >s$. 
\end{itemize}
A semiorthogonal decomposition is called {\em maximal} if the $\cA_s$ do not admit a further nontrivial semiorthogonal decomposition. 
Explicit semiorthogonal decompositions have been computed in many examples, not only over fields but also over more general bases, see, e.g., \cite{kuz-semi}.

\begin{exam}
\label{exam:fano}
Let $X\subset \bP^n$ be a smooth Fano variety with Picard group of rank one, generated by the hyperplane class, and of index $r$. Then there is a semiorthogonal decomposition
$$
\rD^b(X)=\langle \cA_X, \cO_X, \cO_X(1),\ldots, \cO_X(r-1)\rangle;
$$
here $\cA_X$ is called the {\em Kuznetzov component} of $\rD^b(X)$.
For smooth cubic fourfolds $X$, one has $r=3$, and the subcategory $\cA_X$ has {\em some} formal properties of the derived category of a K3 surface. 
\end{exam}

\subsection*{Essential dimension and blowups}
A $k$-linear triangulated category $\mathcal{T}$ is said to be of {\em essential dimension} at most $m$, if $\mathcal{T}$ embeds as a full admissible subcategory into a derived category $\rD^b (Z)$ of a smooth projective variety $Z$ of dimension at most $m$. 
Usually, this definition is applied to a piece in a semiorthogonal decomposition of $\rD^b(X)$.

We now recall Orlov's blowup formula \cite[Theorem 4.3]{orlov-mon}: Let 
$$
q: \tilde{X} = \Bl_Z(X)\to X
$$ 
be a blowup of $X$ in a smooth subvariety $Z$. Then there is a diagram
\centerline{
\xymatrix{
\bP(\mathcal N_Z) \ar@{^{(}->}^j[r] \ar[d]_p & \tilde{X}\ar[d]^q\\
Z \ar@{^{(}->}^i[r] & X 
}
}
and we have a collection of subcategories 
$$
\cA_s:=j_*\left(p^*(\rD^b(Z) \otimes \cO_{\bP(\mathcal N_Z)}(s))\right),
$$
(where all the functors are in the derived sense). Note that $\cA_s$ are all equivalent to $\rD^b(Z)$. 
If  $r:=\mathrm{codim}(Z)$, then 
$$
\langle \cA_{-r+1},\ldots, \cA_{-1}, q^*(\rD^b(X))\rangle
$$
is a semiorthogonal decomposition of $\rD^b(X)$. 

Since $\bP^n$ has a full exceptional collection, and every smooth projective rational variety $X$ can be linked to $\bP^n$ by a sequence of blowups and blowdowns along smooth centers, it has become a guiding principle that such $X$ {\em should} have semiorthogonal decompositions with pieces of essential dimension at most $n-2$. There are various issues that arise, e.g., maximal decompositions are by no means unique, see \cite{BB}, \cite{kuz-simple}, \cite[Remark 5.6]{Orlov-smooth}.  
Still, this point of view has been the basis of conjectures concerning rationality of higher-dimensional varieties 
over closed and nonclosed fields, 
e.g., 
cubic fourfolds \cite[Conj 4.2]{kuz-der-view}, Gushel-Mukai varieties \cite{Perry-Kuz}, and Brauer-Severi varieties, del Pezzo surfaces, Fano threefolds \cite{Bernardara}, \cite{AB-surv}, \cite{auel-bernardara}, \cite{Kuz-Pro}.

\subsection*{$G$-categories}
We turn to $G$-varieties $X$, where $G\subset \Aut(X)$ is a finite group. There is an induced embedding 
$$
G\hookrightarrow \mathrm{Auteq}(\rD^b(X)),
$$
so that $\rD^b(X)$ is a $G$-category. 
The reconstruction theorem of \cite{BO} admits a natural generalization to the equivariant context:

\begin{prop}
Suppose $X$ and $Y$ are smooth projective $G$-varieties, $X$ is Fano, and 
\[
\Phi\colon \rD^b (X) \simeq \rD^b (Y)
\]
is an equivalence of $G$-categories. Then there exists a $G$-equivariant isomorphism
\[
\varphi\colon X \to Y 
\]
inducing $\Phi$. 
\end{prop}

\begin{proof}
The fact that $Y$ is also Fano and $X$ and $Y$ are isomorphic as varieties is just the reconstruction theorem of Bondal and Orlov \cite{BO}. More precisely, one can define a \emph{point object} in $\rD^b (X)$ (and similarly $\rD^b (Y)$) as an object $P$ such that
\begin{enumerate}
    \item $S_X (P) \simeq P [n]$, where $S_X$ is the Serre functor and $n\in \mathbb{Z}$.
    \item For $r<0$ one has $\mathrm{Ext}^r (P, P) =0$.
    \item $\mathrm{Hom}(P,P) \simeq k$.
\end{enumerate}
These conditions imply, if $X$ is Fano, that $n=\dim X$ and $P$ is, up to shift, the skyscraper sheaf $k_{\fp}$ of a closed point $\fp\in X$. Furthermore, from the given equivalence $\Phi$, the reconstruction procedure of Bondal and Orlov outputs an  isomorphism $\varphi$ with the property that if $\Phi$ maps a point object $P$ with support $\fp$ in $X$ to a point object $Q$ with support $\fq$ in $Y$, then $\varphi (\fp)=\fq$. Since $\Phi$ is assumed to be an equivalence of $G$-categories, it follows that $\varphi$ is a $G$-morphism as well. 
\end{proof}

\begin{prop}
Let $X$ be a smooth projective $G$-variety.
Assume that there is a semiorthogonal decomposition
$$
\rD^b(X) = \langle \cA,\mathcal B\rangle,
$$
where $\mathcal B$ is a $G$-stable subcategory. Then $\cA$ is $G$-stable. 

\end{prop}

\begin{proof}
Indeed, $\mathcal{A}$ is, by definition, the full subcategory consisting of all objects $X$ that satisfy
\[
\mathrm{Hom}(Y, X[i])=0, \quad \forall\, i\in \bZ ,
\]
for all $Y$ in $\mathcal{B}$. 
Denoting the action of an element $g\in G$ on an object in $\mathrm{D}^b (X)$ by $g_*$ we obtain
\[
\mathrm{Hom}(g_*Y, g_*X[i])=0, \quad \forall\, i\in \bZ,
\]
because $g$ acts by an autoequivalence on $\mathrm{D}^b (X)$, and in particular, $\mathrm{Hom}(g_*Y, g_*X[i])$ is isomorphic to $\mathrm{Hom}(Y, X[i])$ as $k$-vector space. Since $g_* Y$ is another object of $\mathcal{B}$ and all objects in $\mathcal{B}$ are of this form (because $\mathcal{B}$ is $G$-stable), we get that $g_* X$ is an object of $\mathcal{A}$.
\end{proof}

\begin{coro}
In the notation of Example~\ref{exam:fano}, the Kuznetsov component $\cA_X$ of a $G$-Fano variety is naturally a $G$-category. 
\end{coro}

We recall from \cite{ploog} the notion of a $G$-linearized object of $\rD^b(X)$: 

\begin{defi}\label{def:linearizedobject}
A complex $E^{\bullet}$ in $\rD^b (X)$ is $G$-linearized if it is equipped with a $G$-linearization, i.e., a system of isomorphisms
\[
\lambda_g\colon E^{\bullet} \to g^* E^{\bullet} 
\]
for each $g\in G$, satisfying the compatibility condition 
\[
\lambda_{1} = \mathrm{id}_{E^{\bullet}}, \quad \lambda_{gh}= h^* (\lambda_g)\circ \lambda_h . 
\]
\end{defi}

\subsection*{Nonbirational linear actions}

Let $\fp\in X$ be a closed point and $T_{\fp} X$ the tangent space at $\fp$. For the skyscraper sheaf $k_{\fp}$ we have
\[
\mathrm{Ext}^r (k_{\fp}, k_{\fp}) \simeq \Lambda^r T_{\fp} X, \quad r \in [0, \dim X],
\]
and zero otherwise. In particular, $\mathrm{Ext}^1 (k_{\fp}, k_{\fp})$ parametrizes length $2$ zero-dimensional subschemes supported at $\fp$, which are tangent vectors at $\fp$ to $X$. When $G$ is abelian and $P$ in $\rD^b (X)$ is a point object fixed under $G$, we can consider the weights of the $G$-action on 
\[
\mathrm{Hom} (P, P[1]) = \mathrm{Ext}^1 (P, P),
\]
these are the weights of the $G$-action on $T_{\fp} X$ for the $G$-fixed point ${\fp}\in X$ that is the support of $P$. 

These weights play a role in the computation of the class of the $G$-action in the {\em equivariant Burnside group}, introduced in \cite{KPT} and \cite{BnG}. 
In particular, this formalism allows to distinguish birational types of linear $G$-actions on a variety as simple as $\bP^2$:

\begin{exam}\label{exam:NonBirP2}
Let $G=C_m\times \fS_3$, $m\ge 5$, the product of the cyclic group of order $m$ and the symmetric group on three letters, $V_2$ the standard $2$-dimensional representation of $\fS_3$, and
$k_{\chi}, k_{\chi'}$ 1-dimensional representations of $C_m$ with primitive characters $\chi, \chi'$, $\chi\neq \pm \chi'$.
Then
\[
\bP (k_{\chi}\oplus V_2) , \quad \bP (k_{\chi'}\oplus V_2)
\]
are not $G$-birational to each other. However, both varieties admit
\[
\cO, \cO(1), \cO(2)
\]
as a full (strong) exceptional sequence of $G$-linearized line bundles (albeit with different linearizations on those bundles).
The failure of $G$-birationality  is  proved in \cite[Example 5.3]{KT-struct} and \cite[Section 10]{KT-vector}, using the Burnside formalism of \cite{BnG}. 
\end{exam}

This example indicates that essential information is  contained in the $G$-linearizations of the objects of the collection, respectively, in the attachment functors/nonzero Hom-spaces between the pieces of the decomposition.

\section{Cubic fourfolds: geometry}
\label{sect:geom}

Let $X\subset \bP^5$ be a smooth cubic fourfold, over $k=\bC$. Let $F=F_1(X)$
be the variety of lines of $X$, it is a holomorphic symplectic fourfold {\em deformation equivalent} to $S^{[2]}$, the Hilbert scheme of two points on a K3 surface $S$. 

\subsection*{Rationality}
In this context, there are three main conjectures
concerning the rationality of $X$ (see \cite{Huy} for background, latest results, and references): each of the following conditions is conjectured to be {\em equivalent} to the rationality of $X$.

\begin{enumerate}
\item There is a primitive isometric embedding of Hodge structures  
\begin{equation*}
\label{eqn:**}
\rH^2(S,\bZ)_{\mathrm{pr}}\hookrightarrow  \rH^4(X,\bZ)_{\mathrm{pr}}(1),
\end{equation*}
for some polarized K3 surface $(S, h)$.
\item There is an equivalence of $k$-linear triangulated categories
$$
\rD^b(S) \simeq \cA_X,
$$
for some K3 surface $S$.
\item There is a {\em birationality}  
$$
F_1(X)\sim S^{[2]},
$$
for some K3 surface $S$.
\end{enumerate}
Recall that (1) and (2) are equivalent, by \cite{AT}.
A motivic version has been addressed in \cite{vial}: cubic fourfolds over arbitrary fields with derived equivalent Kuznetsov components have isomorphic Chow motives. 

While there is growing evidence for the validity of these conjectures, some of it based on extensive numerical experiments, the compatibility of these constructions with group actions remained largely unexplored.

\subsection*{Pfaffians}
\label{sect:pfaffian}

We recall some multilinear algebra occurring in the 
construction of Pfaffian cubic fourfolds. Let $V$ be a 
 $k$-vector space of dimension 6,
and consider the nested strata
\[
\mathrm{Gr}(2, V) \subset \mathrm{Pf}(V) \subset \mathbb{P} (\Lambda^2 V),
\]
where $\mathrm{Pf}(V)$ parametrizes skew $6\times 6$ matrices of generic rank $4$ and $\mathrm{Gr}(2, V)$ those of rank $2$. Dually, we also have
\[
\mathrm{Gr}(2, V^*) \subset \mathrm{Pf}(V^*) \subset \mathbb{P} (\Lambda^2 V^*). 
\]
Given a $5$-dimensional subspace $\bP (L )\subset \mathbb{P} (\Lambda^2 V)$, we have an associated $8$-dimensional subspace $\bP (L^{\perp})$ in $\mathbb{P} (\Lambda^2 V^*)$. If 
$$
X = \mathrm{Pf}(V) \cap \bP (L)
$$ 
is smooth then it is a Pfaffian cubic fourfold with associated K3 surface \begin{equation}
\label{eqn:S}
S=\mathrm{Gr}(2, V^*)\cap \bP (L^{\perp}).
\end{equation}
In this context, Conjectures (1), (2) and (3) have been checked for all Pfaffian cubic fourfolds.

\subsection*{Automorphisms}

Actions of a finite group $G$ on $S$ and $X$ induce actions on 
related geometric objects:
\begin{itemize}
\item the punctual 
Hilbert schemes, 
\item the varieties of rational curves on $X$, e.g., $F=F_1(X)$,
\item (polarized) Hodge structures (if $G$-preserves the  polarizations), 
\item derived categories; note that if the $G$-action on $X\subset \bP^5$ arises from a (projectively) linear action on $\bP^5$, then we obtain a natural $G$-action on the Kuznetsov component $\cA_X$. 
\end{itemize}
A useful notion is that of {\em symplectic} automorphisms $G_s\subseteq G$: in the case of K3 surfaces these act trivially on $\rH^{2,0}(S,\bZ)$ and for cubic fourfolds on $\rH^{3,1}(X,\bZ)$. In both cases, there is an exact sequence
$$
1\to G_s\to G\to C_m\to 1.
$$
All finite automorphisms of K3 surfaces have been classified, see \cite{BH}. Symplectic automorphisms of cubic fourfolds have been classified in \cite{LZ}.

The Torelli theorem implies that we have embeddings 
$$
\Aut(S) \hookrightarrow \mathrm O(\rL_S),
$$
$$
\Aut(X) \hookrightarrow \mathrm O(\rL_X),
$$
the group of isometries of the lattices
$$
\rL_S:=\rH^2(S,\bZ)_{\rm pr}, \quad \rL_X:=\rH^4(X,\bZ)_{\rm pr},
$$ 
the latter group coinciding with the group of Hodge isometries of $\rH^4(X,\bZ)$ fixing the polarization. 

In a similar vein, one has injective homomorphisms
$$
\Aut(S)\hookrightarrow \mathrm{Auteq}(\rD^b(S)), \quad 
\Aut(X)\hookrightarrow \mathrm{Auteq}(\cA_X),
$$
into the group of autoequivalences of the corresponding categories, see, e.g., \cite[Theorem 1.3]{Ouchi}. 

Given the naturality of the above constructions, one would expect the following versions of rationality conjectures: 
\begin{itemize}
\item[(1$G$)]
There is a primitive isometric embedding of $G$-Hodge structures  
\begin{equation*}
\label{eqn:**}
\rH^2(S,\bZ)_{\mathrm{pr}}\hookrightarrow  
\rH^4(X,\bZ)_{\mathrm{pr}}(1),
\end{equation*}
for some polarized $G$-K3 surface $(S, h)$.
\item[(2$G$)]
There is a $G$-equivariant 
equivalence of $k$-linear triangulated categories
$$
\rD^b(S) \simeq \cA_X,
$$
for some $G$-K3 surface $S$.
\item[(3$G$)]
There exists a {$G$-equivariant birationality} 
$$
F_1(X)\sim S^{[2]},
$$
for some $G$-K3 surface $S$. 
\end{itemize}
In Section~\ref{sect:examples}, we present counterexamples to all three statements. These are based on a $G$-equivariant Pfaffian construction, in which case both $X$ and $S$ carry {\em compatible} $G$-actions.

\section{Automorphisms and Hodge structures}
\label{sect:hodge}

Let $F=F_1(X)$ be the variety of lines of a smooth cubic fourfold $X\subset \bP^5$. Let 
$$
P\subset F \times X
$$ 
be the universal line/incidence correspondence, with projections 
$$
p\colon P \to F, \quad q\colon P \to X.
$$
By 
\cite{BD85}, we have the Abel-Jacobi map
\[
\alpha \colon \rH^4 (X, \bZ ) \to \rH^2 (F , \bZ)(-1),
\]
where $\alpha= p_* q^*$; here we use Poincar\'{e} duality twice to make sense of $p_*$. This homomorphism is an isomorphism of polarized Hodge structures, with the natural polarization on $X$ and Beauville-Bogomolov form on the Picard group of the holomorphic symplectic variety $F$. 

Given a regular $G$-action on $X$ we obtain a natural $G$-action on $F$, and on the associated Hodge structures.  
As $p, q$ are $G$-morphisms and Poincar\'{e} duality is compatible with the natural $G$-actions on homology and cohomology, $\alpha$ is an isomorphism of $G$-Hodge structures in this case; passing to primitive cohomology we obtain a $G$-equivariant isomorphism of polarized Hodge structures
\begin{equation}
\label{eqn:alpha}
\alpha \colon \rH^4 (X, \bZ )_{\mathrm{pr}} \stackrel{\sim}{\longrightarrow} \rH^2 (F , \bZ)_{\mathrm{pr}}(-1). 
\end{equation}

If $X$ is Pfaffian and $S$ is the associated K3 surface
then we have a {\em birational} isomorphism
\begin{equation}
\label{eqn:phi}
\varphi: S^{[2]} \stackrel{\sim}{\dashrightarrow} F,
\end{equation}
constructed as follows: 
fixing general points 
$$
\mathfrak p, \mathfrak q \in S=\mathrm{Gr}(2, V^*) \cap \bP (L^{\perp})
$$ 
we regard them as $2$-planes in $V^*$ and consider their span, a $4$-plane in $V^*$. The two-forms in $\bP(L)\subset \bP(\Lambda^2 V)$ that are zero on $\mathfrak p + \mathfrak q$ form a line in $X$. This extends to the birational isomorphism \eqref{eqn:phi}, which is an {\em isomorphism} if $S$ does not contain a line and $X$ does not contain a plane, by \cite{BD85}. 

By \cite{HJDG},  $\varphi$ induces a primitive isometric embedding of polarized Hodge structures  
\begin{equation}
\label{eqn:sx}
\rH^2(S,\bZ)_{\mathrm{pr}}\hookrightarrow  \rH^4(X,\bZ)_{\mathrm{pr}}(1)
\end{equation}
since 
\[
\rH^2 (S^{[2]}, \bZ ) \simeq \rH^2 (S, \bZ ) \oplus \bZ \delta, 
\]
as polarized Hodge structures. Here $2\delta$ is the divisor corresponding to length-$2$ non-reduced subschemes of $S$; concretely, 
one has a natural blowup morphism
\[
\epsilon\colon S^{[2]} \to S^{(2)}
\]
resolving the singularities of the second symmetric product 
$S^{(2)}$, the map $\epsilon$ associates to a subscheme its associated zero cycle.

All of the above constructions are obviously $G$-equivariant. 
The following theorem, proved by Brendan Hassett in the Appendix, ensures that  \eqref{eqn:sx} is valid in the $G$-equivariant context as well.

\begin{theo}
\label{thm:hodge}
Let $\phi: Y'\dashrightarrow Y$ be a $G$-equivariant birational map of smooth projective holomorphic symplectic varieties over $k=\bC$. Then 
there exists an isomorphism of $G$-Hodge structures
$$
\psi: \rH^2(Y,\bZ)\to \rH^2(Y',\bZ).
$$
\end{theo}

\section{Automorphisms and Kuznetzov components via stability conditions}
\label{sect:aut-kuz}

We recall results from \cite{Ouchi}, connecting actions of 
automorphisms on derived categories of 
Pfaffian cubic fourfolds with those on associated K3 surfaces. 

A labelled cubic fourfold of discriminant $d$ is a pair $(X, K)$ consisting of a smooth cubic fourfold $X$ and a rank $2$ primitive sublattice $K \subset \rH^{2,2}(X, \bZ)$ containing $h^2$, where $h$ is the hyperplane class, and of discriminant $d=\mathrm{disc}(K)$. The subgroup of labeled automorphisms 
$$
\Aut(X,K):=\{ f\in \Aut(X) \, \mid \, f|_K =1\} \subset \Aut(X)
$$
consists of automorphisms fixing every element of $K$. 
Assume that $d$ satisfies 
\begin{itemize}
\item[(*)] $d>6$ and $d\equiv 0$ or $2 \pmod{6}$,
\item[(**)] $d$ is not divisible by 4,9 or odd primes $p\equiv 2 \pmod{3}$.
\end{itemize}
These conditions are {\em equivalent} to the rationality of $X$, via Conjecture (1),
and imply the existence of an associated K3 surface $S$ such that 
$$
\rH^2(S,\bZ)_{\mathrm{pr}}\hookrightarrow  \rH^4(X,\bZ)_{\mathrm{pr}}(1).
$$
Given any object
$$
\cE\in \rD^b(S\times X)
$$
we obtain in the standard way the Fourier-Mukai functor 
$$
\Phi_{\cE}: \rD^b(S)\to \cA_X
$$
(where we tacitly compose with the projection functor $\rD^b(X) \to \cA_X$ to get to $\cA_X$). 
If $\Phi_{\cE}$ is an equivalence and 
$f\in \Aut(X)$ we get the corresponding autoequivalence
$$
f_{\cE} :=\Phi^{-1}_{\cE} \circ f_* \circ \Phi_{\cE} : \rD^b(S) \to \rD^b(S)
$$
via the diagram

\centerline{
\xymatrix{
\rD^b(S) \ar[d]_{f_{\cE}} \ar[r]^{\Phi_{\cE}}  & \cA_X \ar[d]^{f_*}\\
\rD^b(S) \ar[r]^{\Phi_{\cE}} & \cA_X 
}
}

\

We recall the main theorem from \cite{Ouchi}:

\begin{theo}
\label{thm:oouchi}
For $d$ satisfying {\rm(*)} and {\rm(**)} as above there exists an 
$$
\cE\in \rD^b(S\times X)
$$ such that 
$\Phi_{\cE}$ is an equivalence. Moreover, if we start with an automorphism $f\in \mathrm{Aut}(X, K)$ in the labelled automorphism group $\mathrm{Aut}(X, K)$, then 
$f_{\mathcal{E}}$ is in the image of the natural embedding
\[
\mathrm{Aut}(S, h) \hookrightarrow \mathrm{Auteq}(\mathrm{D}^b (S))
\]
and the induced map 
\[
\mathrm{Aut}(X, K) \to \mathrm{Aut}(S, h)
\]
is an isomorphism. 
\end{theo}

This means that given a $G$-action on a smooth cubic fourfold $X$
fixing the sublattice $K\subset \rH^{2,2}(X,\bZ)$ as above, there exists 
a polarized associated K3 surface $(S,h)$, with a $G$-action on $S$ preserving the polarization $h$, such that 
$\rD^b(S)$ is equivariantly equivalent to $\cA_X$. Note however, that 
there may be nonisomorphic but derived equivalent K3 surfaces. 
Under some assumptions on $G$, the uniqueness of $S$ follows, e.g., if 
the subgroup of symplectic automorphisms $G_s\subseteq G$ is not the trivial group or the cyclic group $C_2$ \cite[Theorem 8.4.]{Ouchi}. However, {\em a priori} it is not guaranteed that different $G$-actions on $S$ are related by an autoequivalence in $\rD^b(S)$.  
There are examples of $G\subset \Aut(S)$ which are not conjugated by automorphisms of $S$ but are conjugated via autoequivalences of $\rD^b(S)$ \cite[Section 8]{HT23}.

\section{Automorphisms and Kuznetsov components via equivariant HPD} 
\label{sect:hpd}

We investigate the Homological Projective Duality (HPD) construction in presence of actions of finite groups $G$, in the special case of the Pfaffian construction, as described in \cite{kuz-pfaffian}. This will allow us to construct a functor that identifies, $G$-equivariantly, the Kuznetsov component of a Pfaffian $G$-cubic fourfold with the derived category of the associated $G$-K3 surface.

We work over an algebraically closed field $k$ of characteristic zero, and adhere to the notation of \cite{Kuz-IHES} and \cite{kuz-pfaffian} (which differs from the notation in \cite{Ouchi} and our notation in other sections). We first explain the general structure of (HPD), following \cite{Kuz-IHES}. 
A {\em Lefschetz decomposition} of a derived category $\rD^b(X)$ is a semiorthogonal decomposition of the form
$$
\rD^b(X) = \langle \cA_0, \cA_1,\ldots, \cA_{i-1}(i-1) \rangle,
$$
where 
$$
0\subset \cA_{i-1}, \ldots, \cA_1\subset \cA_0\subset \rD^b(X)
$$ is a chain of {\em admissible} subcategories of $\rD^b(X)$ \cite[Definition 4.1]{Kuz-IHES}.

Let $V$ be a vector space over $k$ and 
$$
Q\subset \bP(V)\times \bP(V^*)
$$
the incidence quadric. 
An algebraic variety $g:=Y\hookrightarrow \bP(V^*)$ is called {\em projectively dual}
to $f:X\hookrightarrow \bP(V)$, with respect to a fixed Lefschetz decomposition on $X$, if there exists an $\cE\in \rD^b(Q(X,Y))$, where
$$
Y\times_{\bP(V^*)} \cX_1=Q(X,Y):= (X\times Y)\times_{\bP(V)\times \bP(V^*)}Q
$$
and $\cX_1$ is the universal hyperplane section of $X$, 
such that the corresponding kernel functor
$$
\Phi=\Phi_{\cE} : \rD^b(Y)\to \rD^b(\cX_1)
$$
is fully faithful and gives the following semiorthogonal decomposition
$$
\rD^b(\cX_1) = \langle \Phi(\rD^b(Y)), \cA_1(1)\boxtimes \rD^b(\bP(V^*)), \ldots, 
 \cA_{i-1}(i-1)\boxtimes \rD^b(\bP(V^*))
\rangle.
$$
The main result \cite[Theorem 6.3]{Kuz-IHES} says that 
\begin{itemize}
\item if $X$ is smooth then $Y$ is smooth and admits a {\em canonical} dual Lefschetz decomposition, 
\item there is a base-change functor (see \cite[Section 2]{Kuz-IHES}) which allows to restrict these structures to an {\em admissible} linear subspace $L\subset V^*$. In detail, if 
$$
X_L:=X\times_{\bP(V)} \bP(L^\perp), \quad Y_L:=Y\times_{\bP(V^*)} \bP(L)
$$
then there is a {\em canonical} decomposition of their categories.  
\end{itemize}
The main point of \cite{kuz-pfaffian} is to produce  the Fourier-Mukai kernel $\cE$ 
and check the required properties.
We introduce the following actors:
\begin{enumerate}
\item $X:=\mathbf G=\Gr(2,W)$, where $W$ is a 6-dimensional vector space,
\item $\cU\subset W\otimes \cO_X$ the tautological subbundle of rank 2 on $X$, 
\item $\mathbb P := \mathbb P(V^*)$,  $\mathbb P^\vee := \bP(V)$, 
where $V:=\wedge^2 W$, 
\item $Q\subset \bP^\vee \times \bP$, the incidence quadric, 
\item $\cX=\cX_1\subset X\times \mathbb P$ is the universal hyperplane section of $X$,  
\item $Y:=\mathrm{Pf}(W^*)\hookrightarrow \mathbb P$,
\item $\tilde{Y}\subset Y\times \mathbb G$ is the incidence correspondence; we have
$$
\tilde{Y}=\bP_{\mathbf G}(\wedge^2\cK^\perp),
$$
where $\cK\subset W\otimes \cO_{\mathbf G}$ is the tautological subbundle of rank 2, and $\cK^\perp \subset W^*\otimes \cO_{\mathbf G}$ is its orthogonal,
\item the projections 
$$
g_Y:\tilde{Y} \to Y \quad \text{ and } \quad \zeta: \tilde{Y} \to \mathbf G,
$$
\item $\mathcal R:={g_Y}_* \mathcal{E}nd(\cO_{\tilde{Y}} \oplus \cK)$, a sheaf of Azumaya algebras on $Y$, 
\item the incidence quadric
$$
Q(X,\tilde{Y}):= (X\times \tilde{Y})\times_{(\bP^\vee \times \bP)} Q \stackrel{j}{\longrightarrow} \cX\times \tilde{Y}, 
$$
with its embedding, the pullback of $Q$ under the composition
$$
X\times \tilde{Y} \to X\times Y\hookrightarrow \bP(\wedge^2W)\times \bP(\wedge^2W^*),
$$
forming a divisor 
$$
Q(X,\tilde{Y})\subset X\times \tilde{Y},
$$
\item the locus of pairs of intersecting subspaces 
$$
X\times \mathbf G =\Gr(2,W) \times \Gr(2,W)
$$
and its preimage 
$$
T\subset X\times \tilde{Y}
$$
under the natural morphism
$$
(\mathrm{id} \times \zeta) :  X\times \tilde{Y} \ra X\times \mathbf G,
$$
note that 
$$
T\subset Q(X,\tilde{Y}),
$$
\item the bundle 
$$
\cE:=\cJ_{T,Q(X,\tilde{Y})}(H_X+H_{\mathbf G}),
$$
the sheaf of ideals of $T$ in $Q(X,\tilde{Y})$ twisted by the sum of hyperplane classes of $X$ and of $\mathbf G$, 
\end{enumerate}

The main difference to the original HPD is that the role of $Y$ (and its derived category) is now played by the pair $(Y,\cR)$ and 
$$
\rD^b(Y,\cR), 
$$
which is the derived category of coherent sheaves of right $\cR$-modules. By \cite[Theorem 3.2]{kuz-pfaffian}, there is a fully faithful embedding of 
$$
\rD^b(Y,\cR)\hookrightarrow \rD^b(\tilde{Y}),
$$
as an admissible subcategory, with image denoted by $\tilde{D}$. 
 
Consider the coherent sheaf
$$
j_*\cE \in \mathsf{Coh}(\cX\times \tilde{Y}).
$$
By \cite[Lemma 8.2]{kuz-pfaffian}, and the discussion on p. 11 of that paper, we can view $j_*\cE$ as an object in 
$$
\rD^b(X\times Y,\cO_X\boxtimes \cR^{\rm opp}),
$$
derived category of coherent sheaves of right modules over
$\cO_X\boxtimes \cR^{\rm opp}$. 
The associated kernel functor
$$
\Phi_{j_*\cE}:\rD^b(Y,\cR)\to \rD^b(\cX)
$$
is fully faithful, by \cite[Corollary 9.16]{kuz-pfaffian}.

\medskip

First, writing $V=\Lambda^2 W$, we consider the Grassmannians
\[
\mathbf{G}_6 := \mathrm{Gr} (6, V^*), 
\]
parametrizing $6$-dimensional subspaces of $V^*$, with tautological subbundle 
\[
\mathcal{L}_6 \subset V^* \otimes \cO_{\mathbf{G}_6}
\]
and denote by 
\[
\mathcal{L}_6^{\perp} \subset V \otimes \cO_{\mathbf{G}_6}
\]
the orthogonal subbundle. The universal families of linear sections of $X$ of interest to us are 
\begin{align*}
    \mathcal{X}_6 &= (X \times \mathbf{G}_6) \times_{\bP (V) \times \mathbf{G}_6} \mathbb{P}_{\mathbf{G}_6} (\mathcal{L}_6^{\perp}) \\
    \widetilde{\mathcal{Y}}_6 & = (\tilde{Y} \times \mathbf{G}_6) \times_{\bP (V^*) \times \mathbf{G}_6} \mathbb{P}_{\mathbf{G}_6} (\mathcal{L}_6), 
\end{align*}
but actually one considers pairs 
$$
(\cY_6,\cR_6)
$$
consisting of the variety 
$$
\cY_6:=(Y\times \mathbf G_6) \times_{\bP (V^*) \times \mathbb{G}_6} \mathbb{P}_{\mathbf{G}_6} (\mathcal{L}_6), 
$$
together with a sheaf of Azumaya algebras $\cR_6$, obtained by pullback.

These are fibred over $\mathbf{G}_6$; note that the fibres over a (sufficiently general) $L$ are just $X_L$, the K3 surface associated to $Y_L$, a Pfaffian cubic fourfold.  We consider the natural projection
\[
\mathcal{X}_6 \times_{\mathbf{G}_6} \mathcal{Y}_6 \to X \times Y 
\]
and denote by $\mathcal{E}_6$ the pull-back of $j_* \mathcal{E}$, as a sheaf of Azumaya algebras, to $\mathcal{X}_6 \times_{\mathbf{G}_6} \mathcal{Y}_6$, viewed as a sheaf on $\mathcal{X}_6 \times \mathcal{Y}_6$. 

Base-changing to $[L]\in \mathbf{G}_6$ gives similar objects, which we denote by the same symbols as above, but with an added subscript $L$, e.g., 
\[
\mathcal{E}_{6, L}, \quad \mathcal{X}_{6, L} \times \mathcal{Y}_{6, L} = X_L \times Y_L. 
\]
We get the functor 
\[
\Phi_6 \colon \mathrm{D}^b (\mathcal{Y}_6,\cR_6) \to \mathrm{D}^b(\mathcal{X}_6), 
\]
induced by $\mathcal{E}_6$. 
With our notational conventions, we also get functors 
\[
\Phi_{6, L}\colon \mathrm{D}^b (Y_L) \to \mathrm{D}^b(X_L). 
\]
Then Kuznetsov shows, in the commutative context, that both $\Phi_6$ and the $\Phi_{6,L}$ are {\em splitting}, in the sense of \cite[Section 3]{Kuz-IHES}; this uses the faithful base change theorems \cite[Section 2.8]{Kuz-IHES}. The base change theorem holds as well in the context of varieties equipped with sheaves of Azumaya algebras (called {\em Azumaya varieties})
\cite[Section 2.6, and Proposition 2.43]{kuz-hyper}. 
The splitting property of $\Phi_6$, in the context of more general {\em noncommutative varieties}, which include derived categories of Azumaya varieties over fields, is established in \cite[Theorem 8.4]{Perry}. The argument proceeds by induction on dimension of linear sections in the universal families, starting with hyperplanes, and is essentially the one presented for {\em varieties} in \cite[Section 6]{kuz-pfaffian}.

In particular, if we restrict $\Phi_{6, L}$ to the Kuznetsov component $\mathcal{A}_{Y_L}$ in 
\[
\mathrm{D}^b (Y_L, \mathcal{R})=  \mathrm{D}^b (Y_L)= \langle \mathcal{O}(-3), \mathcal{O}(-2), \mathcal{O}(-1), \mathcal{A}_{Y_L} \rangle, 
\]
where $Y_L$ is a smooth Pfaffian cubic fourfold, we obtain an  equivalence  
\[
\mathcal{A}_{Y_L} \simeq\mathrm{D}^b (X_L),
\]
which in this case is the derived category of the associated K3 surface $X_L$. 
All of the above constructions are $G$-equivariant if we endow $W$ with a linear $G$-action. To summarize, we have:

\begin{prop}
\label{prop:kuz-sum}
Let $G$ be a finite group with a faithful 6-dimensional representation $W$. Assume that $\wedge^2(W)^*$ contains a 6-dimensional subrepresentation $L$. Then the functor $\Phi_{6,L}$ induces an equivalence of $G$-categories
\[
\mathcal{A}_{Y_L} \simeq\mathrm{D}^b (X_L).
\]
\end{prop}

\section{Equivariant birational geometry}\label{sect:examples}

In this section, we work over an algebraically closed field $k$ of characteristic zero. 
We write
$$
X\sim_G Y,
$$
when the $G$-varieties $X$ and $Y$ over $k$ are $G$-birational.

Standard examples include {\em linear} or {\em projectively linear} actions of $G$, i.e., generically free actions of $G$ on $\bP^n=\bP(V)$ arising from a linear faithful representation $V$ of $G$, respectively, a linear representation of a central extension of $G$ with center acting trivially on $\bP(V)$. 
Among the main problems in $G$-equivariant birational geometry is to identify: 
\begin{itemize}
\item[({\bf L})]
{\em (projectively) linearizable} actions, i.e., 
$$
X\sim_G \bP(V),
$$
\item[({\bf SL})] 
{\em stably (projectively) linearizable} actions, i.e., 
$$
X\times \bP^m \sim_G \bP(V),
$$
with trivial action on the second factor. 
\end{itemize}

In particular, as a variety, $X$ is {\em (stably) rational} over $k$. 
Note that the same variety, even $\bP^n$, can sometimes be equipped with equivariantly nonbirational actions of the same group. The classification of such actions is ultimately linked to the classification of embeddings of $G$ 
into the Cremona group, up to conjugation in that group.

\subsection*{Nonlinearizable actions on hypersurfaces}

There are many instances when $G$-actions on varieties of dimension $d$ are not (projectively) linearizable for the simple reason that $G$ does not admit (projectively) linear actions on $\bP^{d}$. 

An example of this situation is the
Del Pezzo surface of degree 5, viewed as the moduli space of 5 points on $\bP^1$, with the natural action of $\fS_5$; there are no regular actions of $\fS_5$ on $\bP^2$.

Other examples are, possibly singular, hypersurfaces $$
X\subset \bP(V), \quad \dim(V)=q-1,
$$
for some prime power $q>3$, admitting the action of the Frobenius group
$$
G=\mathrm{AGL}_1(\bF_q),
$$
for the finite field $\bF_q$. Indeed, the smallest faithful representation  of $G$ is its unique irreducible representation $V$,  of dimension $q-1$, so that the $G$-action is not linearizable. 
In many cases, $G$ admits no nontrivial central extensions, so that $G$ does not admit even projectively linear actions on projective spaces of smaller dimension than $\dim(\bP(V))$. 
This is the case for
\begin{itemize}
\item $q=5$ and $X$ the smooth quadric surface given by
\begin{equation}
\label{eqn:quad}
\sum_{i=1}^5 x_i^2 =\sum_{i=1}^5 x_i=0. 
\end{equation}
\item $q=7$ and $X\subset \bP^5$: 
the space of invariants is 1-dimensional, these are the Pfaffian cubic fourfolds 
\begin{equation}
\label{eqn:cubepf}
 x_1^2x_2
 +x_2^2x_3
 +x_3^2x_4
 +x_4^2x_5
 +x_5^2x_6
 +x_1x_6^2
 +\lambda^2(x_1x_3x_5
+x_2x_4x_6),
\end{equation}
smooth for $\lambda \not= 0,\xi, \sqrt{3}\xi$, with $\xi$ a $6$-th root of unity. 
\item  $q=8$ and $X\subset \bP^6$ 
is either the quadric
$$
x_1^2 + x_2^2 + x_3^2 + x_4^2 + x_5^2 + x_6^2 + x_7^2 =0,
$$
or the singular cubic fivefold
$$
x_1x_2x_6 + x_1x_3x_4 + x_1x_5x_7 + x_2x_3x_7 + x_2x_4x_5 + x_3x_5x_6 + x_4x_6x_7 =0.
$$
\item 
$q=9$ 
and $X\subset \bP^7$ is the (singular) quartic
\begin{equation}
\label{eqn:4}
\sum_{i=1}^9 x_i^4 =\sum_{i=1}^9 x_i=0. 
\end{equation}
\end{itemize}

\subsection*{Stably linearizable actions and $G$-Pfaffians}

By \cite{HT-stable}, the $G$-quadric surface \eqref{eqn:quad} is stably linearizable. The proof relied on a $G$-equivariant version of the universal torsor formalism. The Pfaffian construction yields 
stable linearizability in a fundamentally different way:

\begin{theo}
\label{thm:pfaff-stable}
Let $G$ be a finite group and $V$ a faithful representation of $G$ of even dimension $n=2m$. Assume that there exists a $G$-subrepresentation $L\subset \wedge^2V$ of dimension $n$. Let 
$$
X:=\mathrm{Pf}(V)\cap \bP(L), 
$$
and assume that the $G$-action on $X$ is generically free and that the generic rank of the $G$-vector bundle $\cK_X\to X$ is 2. Then 
$X\times \bP^1$, with trivial action on the second factor, is $G$-linearizable. 
\end{theo}

\begin{proof}
Viewing each point $x\in X$ as a skew-symmetric map $V^* \to V$, we let 
$$
K_x\subset V^*
$$
be the kernel of $x$, where $x$ is viewed as a skew-matrix. For general $x$, we have $\dim(K_x)=2$; birationally, this gives a $G$-linearized vector bundle $\mathcal{K}_X \to X$ of rank 2. The No-Name Lemma \cite{BK85}, \cite{Do85}, \cite[\S 3.2, Cor. 3.12]{CS05} implies that $\mathcal{K}_X$ is $G$-birational to $X \times \bA^2$ (where the $G$-action on $\bA^2$ is trivial), moreover for the associated projective bundle  
we have
$$
\mathbb{P}_X(\mathcal{K}_X) \sim_G X\times \bP^1, 
$$
with trivial action on the second factor. 
On the other hand, we have a $G$-equivariant birational map
$$
\mathbb{P}_X(\mathcal{K}_X)\dashrightarrow \mathbb{P}(V^*).
$$ 
Indeed, given a point $[v^*]$ in $\mathbb{P}(V^*)$, its preimage in $\bP_X (\mathcal{K}_X)$ is the set of all skew-symmetric maps in $\bP (L) \subset \bP(\Lambda^2 V)$ containing $v^*$ in their kernel, thus equal to 
$$
\bP (L) \cap \bP (\Lambda^2 (v^*)^{\perp}) \subset \bP ( \Lambda^2 V).
$$
For generic $v^*$, this is a point, by dimension count. 
This shows that $X \times \bP^1$, which is $G$-birational to $\bP_X (\mathcal{K}_X)$, is linearizable. 
\end{proof}

\subsection*{Pfaffian quadrics}

The $G$-Pfaffian formalism yields new results already for  quadric surfaces. 

Let $X=\bP^1\times \bP^1$ and assume that 
the $G$-action on $X$ is 
generically free and minimal, i.e., $\Pic(X)^G=\bZ$. 
Then there is an extension 
$$
1\to G_0 \to G\to C_2\to 1,
$$
where $G_0$ is the intersection of $G$ with the
identity component of $\Aut(\bP^1)^2=\mathrm{PGL}_2^2$, 
and $C_2$ switches the factors in $\bP^1\times \bP^1$. 

The linearizability problem of such actions is settled, see, e.g., \cite{sarikyan}. 
The stable linearizability problem has been settled in  
\cite[Proposition 16]{HT-stable}: the only relevant case is when 
$$
G_0= \mathfrak D_{2n} \times_D \mathfrak D_{2n},
$$
with $D$ the intersection of $G_0$ with the diagonal subgroup, and $n$ {\em odd}. Here the 
dihedral group $\mathfrak D_{2n}$ of order $2n$ acts generically freely on $\bP^1$. 
Then $X\times \bP^2$, with trivial action on the second factor, is linearizable.  
In this situation, we obtain the following improvement:

\begin{prop}
\label{prop:HTQuadrics}
 The quadric surface $X$ is not linearizable but is 
 stably linearizable of level $1$, i.e., $X \times \bP^1$, with trivial action on the second factor, is linearizable. 
\end{prop}

This answers a question raised in \cite[Remark 9.14]{Lemire}, 
strengthening a theorem from 
\cite[Section 9]{Lemire}, in the case $G=C_2\times \fS_3$, 
and \cite[Proposition 16]{HT-stable} in general. 

\begin{proof}
The nonlinearizability statement follows from  \cite{sarikyan}, see the discussion preceding \cite[Proposition 16]{HT-stable}.

Let $W', W''$ be irreducible faithful two-dimensional representations of $\mathfrak D_{2n}$. Put $L= W' \otimes W''$ and $V = W' \oplus W''$. There is a natural $\mathfrak D_{2n}$-invariant skew symmetric matrix $M$ corresponding to
\[
    W' \otimes W'' \subset \Lambda^2(W' \oplus W'')
\]
This matrix is also anti-invariant under the $C_2$ exchanging $W'$ and $W''$. More precisely, if $w_1', w_2'$ is a basis of $W'$ and $w_1'',w_2''$ is a basis of $W''$, we can choose $x_{ij} := w_i'\otimes w_j''$ as a basis of $W'\otimes W''$. In this basis, 
\[
    M = \begin{pmatrix}
        0 & 0 & x_{11} & x_{12} \\
        0 & 0 & x_{21} & x_{22} \\
        -x_{11} & -x_{21} & 0 & 0 \\
        -x_{12} & -x_{22} & 0 & 0 \\    
    \end{pmatrix}.
\]
The induced Pfaffian representation of $X$ satisfies the assumptions of Theorem~\ref{thm:pfaff-stable} if and only if $n$ is odd. 
\end{proof}

\subsection*{Pfaffian cubic fourfolds}
We return to the setup of the equivariant Pfaffian construction in Theorem~\ref{thm:pfaff-stable}. 
Concretely, we proceed as follows: 
\begin{enumerate}
    \item Let $V$ be a $G$-representation of dimension 6. 
    \item Assume there is a decomposition of representations
    $$
    \Lambda^2 V=L \oplus L^{\perp},
    $$
    with $L$ a $6$-dimensional and $L^{\perp}$ a $9$-dimensional $G$-representation.  
    \item
    Assume that the cubic fourfold $X\subset \bP (L)$ is smooth; then the associated 
    K3 surface $S \subset \bP (L^{\perp})$ is also smooth, 
    see e.g. \cite[Lemma 4.4]{kuz-der-view}. 
\end{enumerate}
As immediate consequence of Theorem~\ref{thm:pfaff-stable}, we have: 

\begin{coro}
\label{coro:stab-lin}
Let $G$ be a finite group 
admitting a 6-dimensional faithful representation $V$ over $k$, yielding a Pfaffian cubic fourfold
$X\subset \bP(L)$ as described above. Assume that 
the $G$-action on $X$ is generically free. 
Then $X\times \bP^1$ is $G$-linearizable. 
\end{coro}

In this setting, the obvious rationality construction need no longer work in the $G$-equivariant context, and could thus yield nonlinearizable $G$-actions on cubic fourfolds. With this in mind, we excluded in (1) the existence of a $G$-invariant hyperplane in $\bP(V)$. The following example yields a nonlinearizable action. 

\begin{exam}
\label{exam:main}
Consider the \emph{Frobenius group} 
\[
G :=\mathrm{AGL}_1 (\bF_7) = \bF_7 \rtimes \bF_7^\times = C_7 \rtimes C_6.
\]
We describe its representations:
\begin{enumerate}
    \item 
    There are nonisomorphic 1-dimensional representations 
    $$
     k_{\chi_i}, \quad i=1, \ldots , 6,
    $$
    corresponding to characters of the quotient $C_6$. 
    \item
    There is a single faithful irreducible $6$-dimensional representation $V$, induced from a nontrivial character of $C_7$. 
\end{enumerate}
In particular, $G$ has no faithful representations of dimension $\le 5$. One checks that 
$\wedge^2V$ contains $V$ as a subrepresentation, with multiplicity 2. 
Thus we have a $\bP^1$-worth of choices for a $G$-subrepresentation $L\simeq V$ inside $\Lambda^2 V$, and the general one gives a smooth cubic fourfold.  Concretely, 
the matrix
\[
M_\lambda = \begin{pmatrix}
      0&{-\lambda{x}_{4}}&{-{x}_{2}}&0&{x}_{6}&{-\lambda{x}_{3}}\\
      \lambda{x}_{4}&0&\lambda{x}_{5}&{x}_{3}&0&{-{x}_{1}}\\
      {x}_{2}&{-\lambda{x}_{5}}&0&{-\lambda{x}_{6}}&{-{x}_{4}}&0\\
      0&{-\lambda{x}_{3}}&{x}_{6}&0&\lambda{x}_{1}&{x}_{5}\\
      {-{x}_{6}}&0&{x}_{4}&{-\lambda{x}_{1}}&0&{-\lambda{x}_{2}}\\
      \lambda{x}_{3}&{x}_{1}&0&{-{x}_{5}}&\lambda{x}_{2}&0
\end{pmatrix}
\]
is invariant under 
$$
g: (x_1,x_2,x_3,x_4,x_5,x_6) \mapsto 
(\zeta_7x_1,\zeta_7^5x_2,\zeta_7^4x_3,\zeta_7^6x_4,\zeta_7^2x_5,\zeta_7^3x_6)
$$
(with $\zeta_7$ a primitive $7$th root of unity) and 
$$
h\colon x_i \mapsto -x_{i+1},
$$
which generate $G$. Its Pfaffian is given by
\[
    \lambda \bigl(
 x_1^2x_2
 +x_2^2x_3
 +x_3^2x_4
 +x_4^2x_5
 +x_5^2x_6
 +x_1x_6^2
 +\lambda^2(x_1x_3x_5
+x_2x_4x_6)
\bigr),
\]
which is smooth for $\lambda \not= 0,\xi, \sqrt{3}\xi$ with $\xi$ a $6$-th root of unity; these fourfolds appeared in \cite[Theorem 1.2, Case 7(b)]{LZ}.
The associated K3 surface $S\subset \bP (L^{\perp})$ is also smooth and carries a natural, generically free, $G$-action by construction.  
\end{exam}

\begin{rema}\label{rConnectionToSegre}
One has $\mathrm{AGL}_1 (\bF_7) \subset \mathfrak{S}_7$, and it is possible to write the $\mathfrak{S}_7$-invariant smooth cubic fourfold 
\[
\sum_{i=0}^6 z_i^3 =0 , \quad \sum_{i=0}^6 z_i=0
\]
in the above Pfaffian form: indeed, consider the substitution
\[
	f \colon z_i \mapsto \sum_{j=0}^6 \zeta^{ij} x_j
\]
with $\zeta=\zeta_7$ a primitive 7th root of unity. It satisfies
\[
	 f(z_0+\dots+z_6) = 7 x_0 
\]
and $f(z_0^3+\dots+z_6^3)|_{z_0=0}$ turns out to be equal to 
\begin{gather*}
21(x_2^2x_3+x_1x_3^2+2x_1x_2x_4+x_1^2x_5+x_4x_5^2+x_4^2x_6+2x_3x_5x_6+x_2x_6^2).
\end{gather*}
Cyclically permuting
\[
	x_1 \mapsto x_4 \mapsto x_6 \mapsto x_1
\]
gives a multiple of our equation with $\lambda^2 = 2$.
\end{rema}

\begin{theo}\label{thm:GK3NotLin}
Let $G=\mathrm{AGL}_1 (\bF_7)$ and $X \subset \bP^5$ be
a smooth cubic fourfold constructed in Example~\ref{exam:main}. Then: 
\begin{enumerate}
\item The $G$-action on $X$ is not (projectively) linearizable.
 \item The $G$-action on $X\times \bP^1$, with trivial action on the second factor, is linearizable. 
\item There is a $G$-equivariant primitive embeddings of polarized integral Hodge structures
$$
\rH^2(S,\bZ)_{\rm pr} \hookrightarrow \rH^4(X,\bZ)_{\rm pr}(1).
$$
\item The Kuznetsov component 
$\mathcal{A}_X$ from  the natural semiorthogonal decomposition
\[
\rD^b (X) = \langle \mathcal{A}_X , \mathcal{O}_X, \mathcal{O}_X(1), \mathcal{O}_X (2) \rangle,
\]
is equivalent, as a $G$-category, to $\rD^b(S)$ for the $G$-{\rm K3} surface $S$ obtained in the Pfaffian construction. 
\item
The Fano variety of lines $F_1 (X)$ is $G$-birational to $S^{[2]}$. 
\end{enumerate}
\end{theo}

\begin{proof}
Item (1) follows since $G$ has no faithful $5$-dimensional linear representations, thus the action is not linearizable; and since the Schur multiplier of $G$ is trivial (all Sylow subgroups of $G$ are cyclic), every projective representation of $G$ lifts to a linear representation of $G$. 

\medskip

Item (2) is Corollary \ref{coro:stab-lin}. 

\medskip

Item (3) is proved in the Appendix, by Brendan Hassett. 

\medskip

Item (4) follows from the $G$-equivariance of the functor 
$$
\mathcal{A}_X \to \rD^b (S),
$$ 
given by Kuznetsov's HPD construction; 
we summarized the main ingredients in Section \ref{sect:hpd} 
(with Kuznetsov's notation), see Proposition~\ref{prop:kuz-sum}. 

Alternatively, Ouchi's work \cite{Ouchi}, recalled in Section \ref{sect:aut-kuz}, yields the statement in a similar, although slightly weaker form: first, the $G$-action in our example fixes the sublattice $K\subset \mathrm{H}^{2,2}(X, \bZ)$ spanned by $h^2$ and the \emph{class} of a quintic del Pezzo surface $\Sigma$ in $X$ (but not an actual such cycle $\Sigma$ representing that class!). Points in $X$ can be viewed as skew-symmetric maps $V^* \to V$, and it makes sense to consider the locus of points in $X$ giving skew maps with kernel contained in some fixed chosen five-dimensional subspace $R_5 \subset V^*$. In general, this is a smooth quintic del Pezzo surface $\Sigma =\Sigma_{R_5}$. All such $\Sigma$'s yield the same class in cohomology, in fact, they are all algebraically equivalent (they form one connected algebraic family of cycles in $X$ parametrized by points in the Grassmannian $\mathrm{Gr}(5, V^*)$). In particular, $g(\Sigma )$ and $\Sigma$ give the same class. Ouchi's Theorem \ref{thm:oouchi} (and the subsequent discussion concerning the uniqueness of $S$) imply that in our example, $\cA_X$ is equivalent as a $G$-category to $\mathrm{D}^b (S)$ for  \emph{some} action of the group $G$ on $S$ (but we cannot conclude immediately that it is the one given by the Pfaffian construction). Note that in our case the subgroup of symplectic automorphisms $G_s$ of $G$ cannot be reduced to the trivial group or $C_2$ because these are not subgroups of $G$ with cyclic quotients. 

\medskip

Item (5) follows as  the construction in (\ref{eqn:phi}) is $G$-equivariant. 

\end{proof}

\bibliographystyle{alpha}
\bibliography{derbir}

\section{Appendix, by Brendan Hassett}

Fix a finite group $G$.
Let $Y$ and $Y'$ be projective hyperk\"ahler manifolds with regular $G$-actions.
We assume throughout the existence of a $G$-equivariant birational map
$$\phi: Y'\stackrel{\sim}{\dashrightarrow} Y.$$

Without group actions, Huybrechts \cite[Cor.~4.7]{HuyInventiones} shows that $\phi$ induces an isomorphism of Hodge structures
$$\psi: \mathrm{H}^*(Y,\bZ) \stackrel{\sim}{\rightarrow} 
\mathrm{H}^*(Y',\bZ).$$
Indeed, this follows from a geometric construction \cite[Th.~4.6]{HuyInventiones}:
\begin{itemize}
\item{a connected complex pointed curve $(S,0)$;}
\item{families 
$$
\cY,\cY' \rightarrow S
$$
of smooth hyperk\"ahler manifolds with distinguished fibers
$$Y\simeq \cY_0, \quad Y' \simeq \cY'_0;$$}
\item{an isomorphism 
$$ \Phi:\cY'|_{S\setminus \{0\}} \simeq \cY|_{S\setminus \{0\}}$$
over $S\setminus \{0\}$.}
\end{itemize}
The induced isomorphisms on cohomology yield the desired $\psi$, on specialization to $0$.

Here we explain how to carry out the argument while respecting
the group action.

We record some elementary facts:
\begin{lemm}
Let $\phi$ be a birational map of hyperk\"ahler varieties with $G$-action as above. Then we have
\begin{itemize}
\item{the indeterminacy of $\phi$ and $\phi^{-1}$ has codimension $\ge 2$;}
\item{$\phi$ induces isomorphisms
$$\phi^*:\rH^2(Y,\bR) \stackrel{\sim}{\rightarrow} \rH^2(Y',\bR)$$
whence
$$\phi^*: \Gamma(\Omega^2_Y) \stackrel{\sim}{\rightarrow} 
\Gamma(\Omega^2_{Y'}) \text{ and }
\phi^*: \Pic(Y) \stackrel{\sim}{\rightarrow} 
\Pic(Y'),$$
all compatible with the group action. In particular, the symplectic forms
on $Y$ and $Y'$ yield the same characters of $G$.}
\end{itemize}
\end{lemm}
\begin{proof}
The indeterminacy of our maps is $G$-invariant and has (complex)
codimension $>1$ because both $Y$ and $Y'$ have trivial canonical class.
This precludes any exceptional divisors.  

The isomorphism on cohomology follows from dimensional considerations.
The compatible isomorphisms for holomorphic $2$-forms and the
Picard group reflect Hartogs-type extension theorems.  
\end{proof}

Choose $L'$ to be an ample line bundle on $Y'$ that admits a 
linearization of the $G$-action.  Let $L$ be the corresponding
line bundle on $Y$ under the pull-back homomorphism,  which is
necessarily $G$-invariant as well. Note that the
Beauville-Bogomolov-Fujiki forms $q_Y$ and $q_{Y'}$ take
the same values (see \cite[p.~92]{HuyInventiones})
$$
q_Y(L)=q_{Y'}(L').
$$

The deformation spaces $\operatorname{Def}(Y,L)$ and 
$\Def(Y',L')$ (as polarized varieties) are germs of analytic spaces, with 
tangent spaces
$$L^{\perp} \subset \rH^1(Y,T_Y)\simeq \rH^1(Y,\Omega^1_Y),
\quad
(L')^{\perp} \subset \rH^1(Y',T_{Y'})\simeq 
\rH^1(Y',\Omega^1_{Y'}).$$
These come with natural $G$-actions and equivariant isomorphisms
$$\operatorname{Def}(Y,L) \stackrel{\sim}{\rightarrow} \operatorname{Def}(Y',L').$$

\begin{rema}
The group $G$ may fail to act faithfully on 
$\Def(Y,L)$. The kernel $G_{\circ}\subset G$ acts on fibers of the family \cite{HTMMJ}.
Observe that $G_{\circ} \times \Def(Y,L)$ has a natural
$G$-action
$$g\cdot (g_{\circ},y) = (gg_{\circ}g^{-1},gy)$$
commuting with the fiberwise $G_{\circ}$-action.
\end{rema}

\begin{lemm} \label{lemm:getcurve}
Let $0,\mathfrak{p} \in \Def(Y,L)$ denote the distinguished point and an arbitrary point. There exists a smooth pointed curve $(S,0)$ 
with $G$-action fixing $0$, along with an equivariant morphism
$$(S,0) \rightarrow (\Def(Y,L),0),$$
whose image contains $\mathfrak p$.  
\end{lemm}
\begin{proof}
Consider the universal family 
$$\cY \rightarrow \Def(Y,L)$$
and the diagram

$$\begin{array}{ccccc}
G \times \cY & \rightarrow & G\times^G \cY=\cY &\rightarrow & G \backslash \cY \\
\downarrow &  & \downarrow & & \downarrow \\
G\times \Def(Y,L) &  \rightarrow & G\times^G \Def(Y,L) = \Def(Y,L) &\rightarrow  & G \backslash\Def(Y,L)
\end{array}
$$
When $X$ has a left $G$-action, $G\times^G X$ is the quotient of $G\times X$ under the relation $(hg,x)=(h,gx)$ for $g,h\in G$ and $x\in X$.
The left horizontal arrows are quotients; the right horizontal arrows are induced by projections
onto second factors.

Start with an irreducible curve $S_1$ in the quotient space $G\backslash \operatorname{Def}(Y,L)$ containing the images of $0$ 
and $\mathfrak p$.  
The diagram above and resolution of singularities give a finite morphism $\gamma:S_2 \rightarrow S_1$ from a non-singular curve and a $G$-equivariant morphism
$$\cY_2 \rightarrow S_2$$
such that the classifying morphism
$S_2 \rightarrow G \backslash \operatorname{Def}(Y,L)$ 
coincides with $\gamma$.  
\end{proof}

We choose $\mathfrak p\in\operatorname{Def}(Y,L)$ such that 
$\Pic(\cY_{\mathfrak p})=\bZ L$. Using Lemma~\ref{lemm:getcurve},
choose compatible
$$(S,0) \rightarrow (\operatorname{Def}(Y,L),0),
\quad
(S,0) \rightarrow (\operatorname{Def}(Y',L'),0)$$
so that the corresponding families
$$\cY\rightarrow S, \quad \cY'\rightarrow S$$
have generic Picard rank one. We may repeat the argument of \cite{HuyInventiones};
the birationality construction appears in \cite[\S4]{HuyJDG}.  
Our families are $G$-equivariantly isomorphic over a $G$-invariant non-empty open 
$U\subset S$, hence the fibers have
isomorphic Hodge structures over {\em all} $s\in S$, including $0$.
Thus we obtain $G$-equivariant isomorphisms
$$\rH^*(Y',\bZ) = \rH^*(\cY'_0,\bZ) \simeq \rH^*(\cY_0,\bZ) = \rH^*(Y,\bZ),$$
compatible with Hodge structures.

\end{document}